\documentclass[12pt,reqno]{amsart}
\usepackage{etoolbox}

\makeatletter
\patchcmd\maketitle
  {\uppercasenonmath\shorttitle}
  {}
  {}{}
\patchcmd\maketitle
  {\@nx\MakeUppercase{\the\toks@}}
  {\the\toks@}
  {}
  {}{}
\patchcmd\@settitle{\uppercasenonmath\@title}{\Large}{}{}
\patchcmd\@setauthors
  {\MakeUppercase{\authors}}
  {\authors}
  {}{}
\makeatother
\usepackage{amsmath,amssymb,amsthm}
\usepackage{color}
\usepackage{url}
\usepackage{tikz-cd}
\usepackage[utf8]{inputenc}
\usepackage[T1]{fontenc}
\newtheorem{theorem}{Theorem}[section]

\newtheorem{corollary}{Corollary}[section]

\newtheorem{lemma}{Lemma}[section]
\newtheorem{remark}{Remark}[section]

\numberwithin{equation}{section}
  \newtheorem{thqt}{Theorem}
  
  \newtheorem{lemqt}[thqt]{Lemma}

  \renewcommand{\thethqt}{\Alph{thqt}}
\makeatletter
\@namedef{subjclassname@2020}{%
  \textup{2020} Mathematics Subject Classification}
\makeatother
\usepackage[colorlinks=true]{hyperref}
\hypersetup{urlcolor=blue, citecolor=red , linkcolor= blue}

\begin{document}
\address{$^{[1_a]}$ University of Monastir, Faculty of Economic Sciences and Management of Mahdia, Mahdia, Tunisia}
\address{$^{[1_b]}$ Laboratory Physics-Mathematics and Applications (LR/13/ES-22), Faculty of Sciences of Sfax, University of Sfax, Sfax, Tunisia}
\email{\url{kais.feki@hotmail.com}\,;\,\url{kais.feki@fsegma.u-monastir.tn}}

\address{$^{[2]}$ P.G. Department of Mathematics, Utkal University,
Vanivihar, Bhubaneswar 751004, India.}
\email{\url{satyajitsahoo2010@gmail.com}}



\date{\today}
\author[Kais Feki and Satyajit Sahoo] {\Large{Kais Feki}$^{1_{a,b}}$ and \Large{Satyajit Sahoo}$^{2}$ }
\title[Further inequalities for the $\mathbb{A}$-numerical radius of certain $2 \times 2$ operator matrices]{Further inequalities for the $\mathbb{A}$-numerical radius of certain $2 \times 2$ operator matrices}
\maketitle
\begin{abstract}
Let $\mathbb{A}=
\begin{pmatrix}
A & O \\
O & A \\
\end{pmatrix}
$ be a $2\times2$ diagonal operator matrix whose each diagonal entry is a bounded positive (semidefinite) linear operator $A$ acting on a complex Hilbert space $\mathcal{H}$. In this paper, we derive several $\mathbb{A}$-numerical radius inequalities for $2\times 2$ operator matrices whose entries are bounded with respect to the seminorm induced by the positive operator $A$ on $\mathcal{H}$. Some applications of our inequalities are also given.
\end{abstract}

\section{Introduction and Preliminaries}\label{s1}
Throughout this article, $(\mathcal{H},\langle\cdot,\cdot\rangle)$ stands for a complex Hilbert space with associated norm $\|\cdot\|$. If $\mathcal{M}$ is a given linear subspace of $\mathcal{H}$, then $\overline{\mathcal{M}}$ denotes its closure in the norm topology of $\mathcal{H}$. Further, the orthogonal projection onto a closed subspace $\mathcal{S}$ of $\mathcal{H}$ will be denoted by $P_{\mathcal{S}}$. Let $\mathcal{B}(\mathcal{H})$ denote the $C^{\ast}$-algebra of all bounded linear operators acting on $\mathcal{H}$ with the identity operator $I_{\mathcal{H}}$ (or simply $I$ if no confusion arises). If $T\in\mathcal{B}(\mathcal{H})$, then $\mathcal{N}(T), \mathcal{R}(T)$  and $T^*$ are denoted by the kernel, the range and the adjoint of $T$, respectively.  An operator $T\in \mathcal{B}(\mathcal{H})$ is called positive (semi-definite) $\langle Ax, x\rangle\geq 0$, for every $x\in \mathcal{H}$.
For the rest of this paper, by an operator we mean a bounded linear operator. Further, we suppose that $A\in\mathcal{B}(\mathcal{H})$ is a nonzero positive operator which induces the following semi-inner product
$$\langle\cdot,\cdot\rangle_{A}:\mathcal{H}\times \mathcal{H}\longrightarrow\mathbb{C},\;(x,y)\longmapsto \langle x, y\rangle_{A}:=\langle Ax, y\rangle=\langle A^{1/2}x, A^{1/2}y\rangle.$$
Here $A^{1/2}$ denotes the square root of $A$. Notice that the seminorm induced by ${\langle \cdot, \cdot\rangle}_A$ is given by ${\|x\|}_A=\|A^{1/2}x\|$, for all $x\in\mathcal{H}$. It can checked that ${\|\cdot\|}_A$ is a norm on $\mathcal{H}$ if and only if $A$ is injective, and that the seminormed space $(\mathcal{H}, {\|\cdot\|}_A)$ is complete if and only if $\mathcal{R}(A)$ is a closed subspace of $\mathcal{H}$.

Let $T \in \mathcal{B}(\mathcal{H})$. An operator $S\in\mathcal{B}(\mathcal{H})$ is called an $A$-adjoint of $T$ if $\langle Tx, y\rangle_A=\langle x, Sy\rangle_A$ for all $x,y\in \mathcal{H}$ (see \cite{acg1}). Thus, the existence of an $A$-adjoint of $T$ is equivalent to the existence of a solution of the equation $AX = T^*A$. Notice that this kind of equations can be investigated by using a theorem due to Douglas \cite{doug} which establishes the equivalence between the following  statements:
\begin{itemize}
  \item [(i)] The operator equation $TX=S$ has a bounded linear solution $X$.
  \item [(ii)] $\mathcal{R}(S) \subseteq \mathcal{R}(T)$.
  \item [(iii)] There exists a positive number $\lambda$ such that $\|S^*x\|\leq \lambda \|T^*x\|$ for all $x\in \mathcal{H}$.
\end{itemize}
\noindent Moreover, among many solutions of $AX=S$, it has only one, say $Q$, which satisfies $\mathcal{R}(Q) \subseteq \overline{\mathcal{R}(T^{*})}$. Such $Q$ is said the reduced solution of the equation $TX=S$. If we denote by $\mathcal{B}_{A}(\mathcal{H})$, the subspace of all operators admitting $A$-adjoints, then by Douglas theorem, we  have
$$\mathcal{B}_{A}(\mathcal{H})=\left\{T\in \mathcal{B}(\mathcal{H})\,;\;\mathcal{R}(T^{*}A)\subset \mathcal{R}(A)\right\}.$$
If $T\in \mathcal{B}_A(\mathcal{H})$, the reduced solution of the equation
$AX=T^*A$ will be denoted by $T^{\sharp_A}$. We mention here that, $T^{\sharp_A}=A^\dag T^*A$ in which $A^\dag$ is the Moore-Penrose inverse of $A$ (see \cite{acg2}). In addition, if $T \in \mathcal{B}_A({\mathcal{H}})$, then $T^{\sharp_A} \in \mathcal{B}_A({\mathcal{H}})$, $(T^{\sharp_A})^{\sharp_A}=P_{\overline{\mathcal{R}(A)}}TP_{\overline{\mathcal{R}(A)}}$ and $((T^{\sharp_A})^{\sharp_A})^{\sharp_A}=T^{\sharp_A}$. Moreover, If $S\in \mathcal{B}_A(\mathcal{H})$, then $TS \in\mathcal{B}_A({\mathcal{H}})$ and $(TS)^{\sharp_A}=S^{\sharp_A}T^{\sharp_A}.$ Furthermore, for every $T \in \mathcal{B}_A({\mathcal{H}})$, we have
\begin{equation}\label{diez}
\|T^{\sharp_A}T\|_A = \| TT^{\sharp_A}\|_A=\|T\|_A^2 =\|T^{\sharp_A}\|_A^2.
\end{equation}
 An operator $U\in  \mathcal{B}_A(\mathcal{H})$ is called $A$-unitary if $\|Ux\|_A=\|U^{\sharp_A}x\|_A=\|x\|_A$, for all $x\in \mathcal{H}$. It is worth mentioning that, an operator $U\in  \mathcal{B}_A(\mathcal{H})$ is $A$-unitary if and only if $U^{\sharp_A} U=(U^{\sharp_A})^{\sharp_A} U^{\sharp_A}=P_{\overline{\mathcal{R}(A)}}$ (see \cite{acg1}). For an account of the results, we invite the reader to consult \cite{acg1,acg2}.

An operator $T$ is called $A$-bounded if there exists $\lambda>0$ such that $ \|Tx\|_{A} \leq \lambda \|x\|_{A},\;\forall\,x\in \mathcal{H}.$  By applying Douglas theorem, one can easily see that the subspace of all operators admitting $A^{1/2}$-adjoints, denoted by $\mathcal{B}_{A^{1/2}}(\mathcal{H})$, is equal the collection of all $A$-bounded operators, i.e.,
$$\mathcal{B}_{A^{1/2}}(\mathcal{H})=\left\{T \in \mathcal{B}(\mathcal{H})\,;\;\exists \,\lambda > 0\,;\;\|Tx\|_{A} \leq \lambda \|x\|_{A},\;\forall\,x\in \mathcal{H}  \right\}.$$
Notice that $\mathcal{B}_{A}(\mathcal{H})$ and $\mathcal{B}_{A^{1/2}}(\mathcal{H})$ are two subalgebras of $\mathcal{B}(\mathcal{H})$ which are, in general, neither closed nor dense in $\mathcal{B}(\mathcal{H})$. Moreover, we have $\mathcal{B}_{A}(\mathcal{H})\subset \mathcal{B}_{A^{1/2}}(\mathcal{H})$ (see \cite{acg1,acg3}). Clearly, $\langle\cdot,\cdot\rangle_{A}$ induces a seminorm on $\mathcal{B}_{A^{1/2}}(\mathcal{H})$. Indeed, if $T\in\mathcal{B}_{A^{1/2}}(\mathcal{H})$, then it holds that
\begin{equation}\label{semii}
\|T\|_A:=\sup_{\substack{x\in \overline{\mathcal{R}(A)},\\ x\not=0}}\frac{\|Tx\|_A}{\|x\|_A}=\sup\big\{{\|Tx\|}_A\,; \,\,x\in \mathcal{H},\, {\|x\|}_A =1\big\}<\infty.
\end{equation}

Saddi \cite{saddi} in 2012 defined the $A$-numerical radius of an operator $T\in\mathcal{B}(\mathcal{H})$ by
\begin{align*}
\omega_A(T)
&:= \sup \left\{|\langle Tx, x\rangle_A|\,;\;x\in\mathcal{H},\|x\|_A = 1\right\}.
\end{align*}

Faghih-Ahmadi and Gorjizadeh \cite{fg} in 2016 showed that for $T\in\mathcal{B}_{A^{1/2}}(\mathcal{H})$, we have
\begin{equation}\label{newsemi}
\|T\|_A=\sup\left\{|\langle Tx, y\rangle_A|\,;\;x,y\in \mathcal{H},\,\|x\|_{A}=\|y\|_{A}= 1\right\}.
\end{equation}

We notice here that it may happen that ${\|T\|}_A$ and $\omega_A(T)$ are equal to $+ \infty$ for some $T\in\mathcal{B}(\mathcal{H})$ (see \cite{feki01}). However, ${\|\cdot\|}_A$ and $\omega_A(\cdot)$ are equivalent seminorms on $\mathcal{B}_{A^{1/2}}(\mathcal{H})$. More precisely, In 2018, Baklouti et al. \cite{bakfeki01} showed that  for every $T\in \mathcal{B}_{A^{1/2}}(\mathcal{H})$, we have
\begin{equation}\label{refine1}
\tfrac{1}{2} \|T\|_A\leq\omega_A(T) \leq \|T\|_A.
\end{equation}

For the sequel, for any arbitrary operator $T\in {\mathcal B}_A({\mathcal H})$, we write
$$\Re_A(T):=\frac{T+T^{\sharp_A}}{2}\;\;\text{ and }\;\;\Im_A(T):=\frac{T-T^{\sharp_A}}{2i}.$$
Recently, in 2019 Zamani \cite[Theorem 2.5]{zamani1} showed that if $T\in\mathcal{B}_{A}(\mathcal{H})$, then
\begin{align}\label{zm}
\omega_A(T) = \displaystyle{\sup_{\theta \in \mathbb{R}}}{\left\|\Re_A(e^{i\theta}T)\right\|}_A=\displaystyle{\sup_{\theta \in \mathbb{R}}}{\left\|\Im_A(e^{i\theta}T)\right\|}_A.
\end{align}
Notice that \eqref{zm} is also proved in a general context in \cite{bm}. In 2020, the concept of the $A$-spectral radius of $A$-bounded operators was introduced by the first author in \cite{feki01} as follows:
\begin{equation}\label{newrad}
r_A(T):=\displaystyle\inf_{n\geq 1}\|T^n\|_A^{\frac{1}{n}}=\displaystyle\lim_{n\to\infty}\|T^n\|_A^{\frac{1}{n}}.
\end{equation}
Here we want to mention that the proof of the second equality in \eqref{newrad} can also be found in \cite[Theorem 1]{feki01}. Like the classical spectral radius of Hilbert space operators, it was shown in \cite{feki01} that $r_A(\cdot)$ satisfies the commutativity property, i.e.
\begin{equation}\label{commut}
r_A(TS)=r_A(ST),
\end{equation}
for all $T,S\in \mathcal{B}_{A^{1/2}}(\mathcal{H})$. For the sequel, if $A=I$, then $\|T\|$, $r(T)$ and $\omega(T)$ denote respectively the classical operator norm, the spectral radius and the numerical radius of an operator $T$.

An operator $T\in\mathcal{B}(\mathcal{H})$ is called $A$-selfadjoint if $AT$ is selfadjoint, i.e., $AT = T^*A$ and it is called $A$-positive if $AT\geq0$. If $T$ is $A$-positive, we will write $T\geq_{A}0$. In recent years, several results covering some classes of operators on a complex Hilbert space $\big(\mathcal{H}, \langle \cdot, \cdot\rangle\big)$ were extended to $\big(\mathcal{H}, {\langle \cdot, \cdot\rangle}_A\big)$. Of course, the extension is not trivial since many difficulties arise. For instance, as mentioned above, it may happen
that ${\|T\|}_A = \infty$ for some $T\in \mathcal{B}(\mathcal{H})$. Moreover, no operator admits an adjoint operator for the
semi-inner product ${\langle \cdot, \cdot\rangle}_A$. In addition, for $T \in \mathcal{B}_A({\mathcal{H}})$, we have $(T^{\sharp_A})^{\sharp_A}=P_{\overline{\mathcal{R}(A)}}TP_{\overline{\mathcal{R}(A)}}\neq T$. For further details about $A$-numerical radius, interested readers can follow \cite{bakna,bakfeki01,bakfeki04,BPl,ConFe,bfeki,feki03,fekilaa,zamani1,NSD} and the references therein.

In this paper, we consider the ${2\times 2}$ operator diagonal matrix $\mathbb{A}=\begin{pmatrix}
A &O\\
O &A
\end{pmatrix}$. Clearly, $\mathbb{A}\in \mathcal{B}(\mathcal{H}\oplus \mathcal{H})^+$. So, $\mathbb{A}$ induces the following semi-inner product
$$\langle x, y\rangle_{\mathbb{A}}= \langle \mathbb{A}x, y\rangle=\langle x_1, y_1\rangle_A+\langle x_2, y_2\rangle_A,$$
for all $x=(x_1,x_2)\in \mathcal{H}\oplus \mathcal{H}$ and $y=(y_1,y_2)\in \mathcal{H}\oplus \mathcal{H}$.

Recently, several inequalities for the $\mathbb{A}$-numerical radius of $2 \times 2$ operator matrices have been established by Bhunia et al.  when $A$ is a positive injective operators (see \cite{BPN}). Moreover, different upper and lower bounds of $\mathbb{A}$-numerical radius when $A$ is a positive semidefinite operator has been recently investigated by the first author in \cite{feki04}, by Rout et al. in \cite{rout} and by Kittaneh et al. in \cite{KITSAT}. In this article, we will continue working in this direction and we will prove several new $\mathbb{A}$-numerical radius inequalities of certain $2 \times 2$ operator matrices. The inspiration for our investigation comes from \cite{dolat,HirKit,HirKit2}.

 \section{Results}\label{s2}
In this section, we present our results. Throughout this section $\mathbb{A}$ is denoted to be the $2\times 2$ operator diagonal matrix whose each diagonal entry is the positive operator $A$. To prove our first result, the following lemmas are required.

\begin{lemma}(\cite{feki01})\label{ll2020}
Let $T\in\mathcal{B}(\mathcal{H})$ is an $A$-self-adjoint operator. Then,
\begin{equation*}
\|T\|_{A}=\omega_A(T)=r_A(T).
\end{equation*}
\end{lemma}

\begin{lemma}\label{lm5}(\cite{feki02})
Let $\mathbb{T}=\begin{pmatrix}
T_{11}&T_{12} \\
T_{21}&T_{22}
\end{pmatrix}$ be such that $T_{ij}\in \mathcal{B}_{A^{1/2}}(\mathcal{H})$ for all $i,j\in\{1,2\}$. Then, $\mathbb{T}\in \mathcal{B}_{\mathbb{A}^{1/2}}(\mathcal{H}\oplus \mathcal{H})$ and
$$ r_\mathbb{A}\left(\mathbb{T}\right)\leq r\left[\begin{pmatrix}
\|T_{11}\|_A & \|T_{12}\|_A \\
\|T_{21}\|_A & \|T_{22}\|_A
\end{pmatrix}\right].$$
\end{lemma}
\begin{lemma}(\cite{bfeki,feki04})\label{lemma1}
Let $P, Q, R, S\in\mathcal{B}_{A^{1/2}}(\mathcal{H})$. Then, the following assertions hold
\begin{itemize}
\item[(i)] $\omega_{\mathbb{A}}\left[\begin{pmatrix}
P & O \\
O & S
\end{pmatrix}\right] = \max\big\{\omega_{A}(P), \omega_{A}(S)\big\}$.
\item[(ii)] ${\left\|\begin{pmatrix}
P & O \\
O & S
\end{pmatrix}\right\|}_{\mathbb{A}} = {\left\|\begin{pmatrix}
O & P \\
S & O
\end{pmatrix}\right\|}_{\mathbb{A}} = \max\big\{{\|P\|}_{A}, {\|S\|}_{A}\big\}$.
\item[(iii)] If $P, Q, R, S\in\mathcal{B}_{A}(\mathcal{H})$, then ${\begin{pmatrix}
P & Q \\
R & S
\end{pmatrix}}^{\sharp_{\mathbb{A}}} = \begin{pmatrix}
P^{\sharp_A} & R^{\sharp_A} \\
Q^{\sharp_A} & S^{\sharp_A}
\end{pmatrix}$.
\end{itemize}
\end{lemma}
 Now, we are ready to prove our first result which generalizes \cite[Theorem 2.7]{dolat}.
\begin{theorem}\label{thm101}
Let  $P, Q, R, S\in \mathcal{B}_{A}(\mathcal{H})$. Then, for $\lambda\in [0, 1]$, we have
\begin{align*}
&\omega_{\mathbb{A}}\left[\begin{pmatrix}
P&Q \\
R&S
\end{pmatrix}\right]\\
&\leq \frac{1}{2}\left(\|P\|_A+2\omega_A(S)+\sqrt{\|\lambda^2 PP^{\sharp_A}+QQ^{\sharp_A}\|_A}+\sqrt{\|(1-\lambda)^2 PP^{\sharp_A}+R^{\sharp_A}R\|_A}\right).
\end{align*}
\end{theorem}
\begin{proof}
Let $\mathbb{T}=\begin{pmatrix}
P&Q \\
R&S
\end{pmatrix}$. It is not difficult to see that $\Re_\mathbb{A}(e^{i\theta}\mathbb{T})$ is an $\mathbb{A}$-selfadjoint operator. So, by Lemma \ref{ll2020} we have
$${\left\|\Re_\mathbb{A}(e^{i\theta}\mathbb{T})\right\|}_\mathbb{A}=\omega_\mathbb{A}\Big(\Re_\mathbb{A}(e^{i\theta}\mathbb{T})\Big).$$
So, by applying Lemma \ref{lemma1} (i), we see that
\begin{align*}
&{\left\|\Re_\mathbb{A}(e^{i\theta}\mathbb{T})\right\|}_\mathbb{A}\\
&=\frac{1}{2}\omega_{\mathbb{A}}\left[\begin{pmatrix}
e^{i\theta}P+e^{-i\theta}P^{\sharp_A}&e^{i\theta}Q+e^{-i\theta}R^{\sharp_A} \\
e^{i\theta}R+e^{-i\theta}Q^{\sharp_A}&e^{i\theta}S+e^{-i\theta}S^{\sharp_A}
\end{pmatrix}\right]\\
&\leq\frac{1}{2}\omega_{\mathbb{A}}\left[\begin{pmatrix}
\lambda (e^{i\theta}P+e^{-i\theta}P^{\sharp_A})&e^{i\theta}Q \\
e^{-i\theta}Q^{\sharp_A}&0
\end{pmatrix}\right]\\
&+\frac{1}{2}\omega_{\mathbb{A}}\left[\begin{pmatrix}
(1-\lambda) (e^{i\theta}P+e^{-i\theta}P^{\sharp_A})&e^{-i\theta}R^{\sharp_A} \\
e^{i\theta}R&0
\end{pmatrix}\right]+\frac{1}{2}\omega_{\mathbb{A}}\left[\begin{pmatrix}
0&0 \\
0&e^{i\theta}S+e^{-i\theta}S^{\sharp_A}
\end{pmatrix}\right]\\
&\leq\frac{1}{2}\omega_{\mathbb{A}}\left[\begin{pmatrix}
\lambda (e^{i\theta}P+e^{-i\theta}P^{\sharp_A})&e^{i\theta}Q \\
e^{-i\theta}Q^{\sharp_A}&0
\end{pmatrix}\right]\\
&\;\;+\frac{1}{2}\omega_{\mathbb{A}}\left[\begin{pmatrix}
(1-\lambda) (e^{i\theta}P+e^{-i\theta}P^{\sharp_A})&e^{-i\theta}R^{\sharp_A} \\
e^{i\theta}R&0
\end{pmatrix}\right]+\omega_A(S),
\end{align*}
where the last inequality follows by Lemma \ref{lemma1} (i) together with the triangle inequality. Moreover, it can be observed that
\begin{align*}
\mathbb{A}\begin{pmatrix}
\lambda (e^{i\theta}P+e^{-i\theta}P^{\sharp_A})&e^{i\theta}Q \\
e^{-i\theta}Q^{\sharp_A}&0
\end{pmatrix}&=\begin{pmatrix}
\lambda (e^{i\theta}AP+e^{-i\theta}AP^{\sharp_A})&e^{i\theta}AQ \\
e^{-i\theta}AQ^{\sharp_A}&0
\end{pmatrix}\\
&=\begin{pmatrix}
\lambda \left(e^{i\theta}(P^{\sharp_A})^*A+e^{-i\theta}P^*A\right)&e^{i\theta}(Q^{\sharp_A})^*A \\
e^{-i\theta}Q^*A&0
\end{pmatrix}\\
&=\begin{pmatrix}
\lambda (e^{-i\theta}P^{\sharp_A}+e^{i\theta}P)&e^{i\theta}Q \\
e^{-i\theta}Q^{\sharp_A}&0
\end{pmatrix}^*\mathbb{A}.
\end{align*}
Hence $\begin{pmatrix}
\lambda (e^{i\theta}P+e^{-i\theta}P^{\sharp_A})&e^{i\theta}Q \\
e^{-i\theta}Q^{\sharp_A}&0
\end{pmatrix}$ is $\mathbb{A}$-selfadjoint operator. Similarly one can show that $\begin{pmatrix}
(1-\lambda) (e^{i\theta}P+e^{-i\theta}P^{\sharp_A})&e^{-i\theta}R^{\sharp_A} \\
e^{i\theta}R&0
\end{pmatrix}$ is $\mathbb{A}$-selfadjoint operator. So by applying Lemma \ref{ll2020} we see that
\begin{align*}
{\left\|\Re_\mathbb{A}(e^{i\theta}\mathbb{T})\right\|}_\mathbb{A}
&\leq\frac{1}{2}r_{\mathbb{A}}\left[\begin{pmatrix}
\lambda (e^{i\theta}P+e^{-i\theta}P^{\sharp_A})&e^{i\theta}Q \\
e^{-i\theta}Q^{\sharp_A}&O
\end{pmatrix}\right]\\
&\;\;+\frac{1}{2}r_{\mathbb{A}}\left[\begin{pmatrix}
(1-\lambda) (e^{i\theta}P+e^{-i\theta}P^{\sharp_A})&e^{-i\theta}R^{\sharp_A} \\
e^{i\theta}R&O
\end{pmatrix}\right]+\omega_A(S).
\end{align*}
So, by using \eqref{commut} we infer that
\begin{align*}
{\left\|\Re_\mathbb{A}(e^{i\theta}\mathbb{T})\right\|}_\mathbb{A}
&\leq \frac{1}{2}r_{\mathbb{A}}\left[\begin{pmatrix}
e^{-i\theta}I&O \\
\lambda P&Q
\end{pmatrix}\begin{pmatrix}
\lambda P^{\sharp_A}&e^{i\theta}I \\
Q^{\sharp_A}&O
\end{pmatrix}\right]\\
&\;\;+\frac{1}{2}r_{\mathbb{A}}\left[\begin{pmatrix}
e^{-i\theta}I&O \\
(1-\lambda) P&e^{-2i\theta}R^{\sharp_A}
\end{pmatrix}\begin{pmatrix}
(1-\lambda) P^{\sharp_A}&e^{i\theta}I \\
e^{2i\theta}R&O
\end{pmatrix}\right]+\omega_A(S)
\\
&\leq \frac{1}{2}r_{\mathbb{A}}\left[\begin{pmatrix}
\lambda e^{-i\theta}P^{\sharp_A} & I \\
\lambda^2 PP^{\sharp_A}+QQ^{\sharp_A}& \lambda e^{i\theta}P
\end{pmatrix}\right]\\
&\;\;+\frac{1}{2}r_{\mathbb{A}}\left[\begin{pmatrix}
(1-\lambda) e^{-i\theta}P^{\sharp_A} & I \\
(1-\lambda)^2 PP^{\sharp_A}+R^{\sharp_A}R&(1-\lambda) e^{i\theta}P
\end{pmatrix}\right]+\omega_A(S).
\end{align*}
So, by using Lemma \ref{lm5} we get
\begin{align*}
&{\left\|\Re_\mathbb{A}(e^{i\theta}\mathbb{T})\right\|}_\mathbb{A}\\
&\leq \frac{1}{2}r\left[\begin{pmatrix}
\lambda \|P\|_A & 1 \\
\|\lambda^2 PP^{\sharp_A}+QQ^{\sharp_A}\|_A& \lambda \|P\|_A
\end{pmatrix}\right]\\
&\;\;+\frac{1}{2}r\left[\begin{pmatrix}
(1-\lambda) \|P\|_A & 1 \\
\|(1-\lambda)^2 PP^{\sharp_A}+R^{\sharp_A}R\|_A&(1-\lambda) \|P\|_A
\end{pmatrix}\right]+\omega_A(S)
\\
&=\frac{1}{2}\bigg[\|P\|_A+2\omega_A(S)+\sqrt{\|\lambda^2 PP^{\sharp_A}+QQ^{\sharp_A}\|_A}+\sqrt{\|(1-\lambda)^2 PP^{\sharp_A}+R^{\sharp_A}R\|_A}\bigg].
\end{align*}
So, by taking the supremum over all $\theta\in\mathbb{R}$ in the last inequality and then using \ref{zm} we get desired result.
\end{proof}

The following Lemma is useful in the sequel. Notice that its assertions generalize recent results done by Rout et al. in \cite{rout} for operators in $\mathcal{B}_{A}(\mathcal{H})$.
\begin{lemma}\label{lem100}
	Let $T,S\in \mathcal{B}_{A^{1/2}}(\mathcal{H})$. Then,
\begin{itemize}
  \item [(i)] $ \omega_\mathbb{A}\left[\begin{pmatrix}
	T&S\\
	S&T
	\end{pmatrix}\right]= \max\{\omega_A(T+S),\omega_A(T-S)\}.$
  \item [(ii)] $ \omega_\mathbb{A}\left[\begin{pmatrix}
	T&-S\\
	S&T
	\end{pmatrix}\right]= \max\{\omega_A(T+iS),\omega_A(T-iS)\}.$
\end{itemize}
\end{lemma}
In order to prove Lemma \ref{lem100} we need the following result.
\begin{lemqt}\label{weak}(\cite{bfeki})
Let $T\in \mathcal{B}_{A^{1/2}}(\mathcal{H})$. Then,
\begin{equation*}
\omega_A(U^{\sharp_A}TU)=\omega_A(T),
\end{equation*}
for any $A$-unitary operator $U\in\mathcal{B}_A(\mathcal{H})$.
\end{lemqt}
Now, we state the proof of Lemma \ref{lem100}.
\begin{proof}[Proof of Lemma \ref{lem100}]
\noindent (i)\;Let $\mathbb{U}=\frac{1}{\sqrt{2}}\begin{pmatrix}
	I & I\\
	-I & I
	\end{pmatrix}$. By using Lemma \ref{lemma1} (iii), we see that
$$\mathbb{U}^{\sharp_\mathbb{A}}=\frac{1}{\sqrt{2}}\begin{pmatrix}
	P_{\overline{\mathcal{R}(A)}} &-P_{\overline{\mathcal{R}(A)}}\\
	P_{\overline{\mathcal{R}(A)}} &P_{\overline{\mathcal{R}(A)}}
	\end{pmatrix}.$$
So, by using the fact that $AP_{\overline{\mathcal{R}(A)}}=A$, we can verify that $\|\mathbb{U}x\|_{\mathbb{A}}=\|\mathbb{U}^{\sharp_{\mathbb{A}}}x\|_{\mathbb{A}}=\|x\|_{\mathbb{A}}$ for all $x\in \mathcal{H}\oplus \mathcal{H}$. Hence $\mathbb{U}$ is $\mathbb{A}$-unitary. So, by Lemma \ref{weak} we see that
	\begin{align*}
	\omega_\mathbb{A}\left[\begin{pmatrix}
	T&S\\
	S&T
	\end{pmatrix}\right]
&= \omega_\mathbb{A}\left[\mathbb{U}^{\sharp_\mathbb{A}}\begin{pmatrix}
	T&S\\
	S&T
	\end{pmatrix}\mathbb{U}\right]\\
	&=\omega_{\mathbb{A}}\left[\begin{pmatrix}
	P_{\overline{\mathcal{R}(A)}} &O\\
	O &P_{\overline{\mathcal{R}(A)}}
	\end{pmatrix}
	\begin{pmatrix}
	T-S &O\\
	O&T+S
	\end{pmatrix}\right]\\
	&=\omega_{\mathbb{A}}\left[
	\begin{pmatrix}
	T-S &O\\
	O&T+S
	\end{pmatrix}\right]\\
	&=\max\{\omega_A(T+S),\omega_A(T-S)\},
	\end{align*}
where the last equality follows from Lemma \ref{lemma1} (i).
\par \vskip 0.1 cm \noindent (ii)\;By considering the $\mathbb{A}$-unitary operator $\mathbb{U}=\frac{1}{\sqrt{2}}\begin{pmatrix}
I & iI\\
iI & I
\end{pmatrix}$ and proceeding as above we get the desired result.
\end{proof}

As an application of Theorem \ref{thm101} together with Lemma \ref{lem100}, we state the following corollary.
\begin{corollary}
Let $P,Q\in \mathcal{B}_{A}(\mathcal{H})$. Then for all $\lambda\in [0,1]$, we have
\[\omega_A\left(P\pm Q\right)\leq \min\{\mu,\nu \},\]
where
\begin{align*}
\mu
& =\omega_A(P)+\frac{1}{2}\left(\|P\|_A+\sqrt{\|\lambda^2 PP^{\sharp_A}+QQ^{\sharp_A}\|_A}+\sqrt{\|(1-\lambda)^2 PP^{\sharp_A}+Q^{\sharp_A}Q\|_A}\right),
\end{align*}
and
\begin{align*}
\nu
& =\frac{3}{2}\omega_A(P)+\frac{1}{2}\bigg[\sqrt{\lambda^2 \omega^2_A(P)+\|Q\|_A^2}
+\sqrt{(1-\lambda)^2 \omega^2_A(P)+\|Q\|_A^2}\bigg].
\end{align*}
\end{corollary}
\begin{proof}
Recall first that it has been recently proved in \cite{BPN1} that
\begin{align}\label{ineq106}
&\omega_{\mathbb{A}}\left[\begin{pmatrix}
P&Q \\
R&S
\end{pmatrix}\right] \nonumber \\
&\leq \frac{1}{2}\bigg[\omega_A(P)+2\omega_A(S)+\sqrt{\lambda^2 \omega^2_A(P)+\|Q\|_A^2}
+\sqrt{(1-\lambda)^2 \omega^2_A(P)+\|R\|_A^2}\bigg].
\end{align}
So, we get the desired result by applying Theorem \ref{thm101} together with \eqref{ineq106} and Lemma \ref{lem100}
\end{proof}
Now we state the following theorem which generalizes a recent result proved by Rout et al. in \cite{rout} since $\mathcal{B}_{A}(\mathcal{H})$ is in general a proper subspace of $\mathcal{B}_{A^{1/2}}(\mathcal{H})$.
\begin{theorem}\label{them10}
Let $\mathbb{T}=\begin{pmatrix}
P&Q \\
R&S
\end{pmatrix}$ be such that $P, Q, R, S\in \mathcal{B}_{A^{1/2}}(\mathcal{H})$. Then,
\begin{align}\label{upper1}
\omega_{\mathbb{A}}\left[\begin{pmatrix}
P&Q \\
R&S
\end{pmatrix}\right]
&\leq \max\{\omega_A(P), \omega_A(S)\}+\frac{\omega_A(Q+R)+\omega_A(Q-R)}{2}.
\end{align}
\end{theorem}
\begin{proof}
	Using triangle inequality we have
\begin{align}\label{sho}
\omega_{\mathbb{A}}\left[\begin{pmatrix}
P&Q \\
R&S
\end{pmatrix}\right]&\leq \omega_{\mathbb{A}}\left[\begin{pmatrix}
P&O \\
O&S
\end{pmatrix}\right]+\omega_{\mathbb{A}}\left[\begin{pmatrix}
O&Q \\
R&O
\end{pmatrix}\right]\nonumber\\
&=\max\{\omega_A(P), \omega_A(S)\}+\omega_{\mathbb{A}}\left[\begin{pmatrix}
O&Q \\
R&O
\end{pmatrix}\right],
\end{align}
where the last equality follows by Lemma \ref{lemma1} (i). 	Let $\mathbb{U}=\frac{1}{\sqrt{2}}\begin{pmatrix}
	I & -I\\
	I & I
	\end{pmatrix}$. By proceeding as above we prove that $\mathbb{U}$ is $\mathbb{A}$-unitary. So, by Lemma \ref{weak} we see that
	\begin{align*}
	\omega_\mathbb{A}\left[\begin{pmatrix}
	O&Q \\
	R&O
	\end{pmatrix}\right]
&= \omega_\mathbb{A}\left[\mathbb{U}^{\sharp_\mathbb{A}}\begin{pmatrix}
	O&Q \\
	R&O
	\end{pmatrix}\mathbb{U}\right]\nonumber\\
	&=\frac{1}{2}\omega_{\mathbb{A}}\left[\begin{pmatrix}
	P_{\overline{\mathcal{R}(A)}} &O\\
	O &P_{\overline{\mathcal{R}(A)}}
	\end{pmatrix}
	\begin{pmatrix}
	R+Q &-R+Q\\
	R-Q &-R-Q
	\end{pmatrix}\right]\nonumber\\
	&=\frac{1}{2}\omega_{\mathbb{A}}\left[
	\begin{pmatrix}
	R+Q &-R+Q\\
	R-Q &-R-Q
	\end{pmatrix}\right]\nonumber\\
	&\leq\frac{1}{2}\omega_{\mathbb{A}}\left[
	\begin{pmatrix}
	R+Q &O\\
	O &-R-Q
	\end{pmatrix}\right]
+\frac{1}{2}\omega_{\mathbb{A}}\left[
	\begin{pmatrix}
	O &-R+Q\\
	R-Q &O
	\end{pmatrix}\right]
\nonumber\\
	&=\frac{1}{2}\Big(\omega_A(Q+R)+\omega_A(Q-R)\Big),
	\end{align*}
where the last equality follows from Lemma \ref{lem100}. So,
\begin{equation}\label{qq2}
	\omega_\mathbb{A}\left[\begin{pmatrix}
	O&Q \\
	R&O
	\end{pmatrix}\right]\leq \frac{1}{2}\Big(\omega_A(Q+R)+\omega_A(Q-R)\Big).
\end{equation}
Combining \eqref{qq2} together with \eqref{sho} yields to the desired result.
\end{proof}

Our next objective is to present an improvement of the inequality \eqref{upper1}. To do this, we need the following lemma.

\begin{lemma}\label{p4200}(\cite{HORN})	Let $T=[t_{ij}]\in M_n(\mathbb{C})$ be such that $t_{ij}\geq 0$ for all $i, j=1, 2,\ldots, n$. Then
	$$\omega(T)= \frac{1}{2}r([t_{ij}+t_{ji}]).$$
\end{lemma}

\begin{theorem}\label{them100}
Let $\mathbb{T}=\begin{pmatrix}
P&Q \\
R&S
\end{pmatrix}$ be such that $P, Q, R, S\in \mathcal{B}_{A}(\mathcal{H})$. Then,
\begin{align*}
&\omega_{\mathbb{A}}\left[\begin{pmatrix}
P&Q \\
R&S
\end{pmatrix}\right]\\
&\leq \frac{1}{2}\bigg(\omega_A(P)+\omega_A(S)+\sqrt{(\omega_A(P)-\omega_A(S))^2+(\omega_A(Q+R)+\omega_A(Q-R))^2}\bigg).
\end{align*}
Moreover, the inequality is sharper than the inequality \eqref{upper1}.
\end{theorem}
\begin{proof}
It follows from \cite[Theorem 2.3]{feki04} that
	\begin{align*}
	\omega_\mathbb{A}\left(\mathbb{T}\right)
&\leq\omega\left[\begin{pmatrix}
	\omega_\mathbb{A}\left({P}\right) & \omega_{\mathbb{A}}\begin{pmatrix}
	O&Q \\
	R&O
	\end{pmatrix}\\
	\omega_{\mathbb{A}}\begin{pmatrix}
	O&Q \\
	R&O
	\end{pmatrix} & \omega_\mathbb{A}\left({S}\right)
	\end{pmatrix} \right] \\
	&=r\left[\begin{pmatrix}
	\omega_\mathbb{A}\left({P}\right) & \omega_{\mathbb{A}}\begin{pmatrix}
	O&Q \\
	R&O
	\end{pmatrix}\\
	\omega_{\mathbb{A}}\begin{pmatrix}
	O&Q \\
	R&O
	\end{pmatrix} & \omega_\mathbb{A}\left({S}\right)
	\end{pmatrix}\right]\\
	&=\frac{1}{2}\bigg[\omega_A(P)+\omega_A(S)+\sqrt{(\omega_A(P)-\omega_A(S))^2+4\omega^2_{\mathbb{A}}\begin{pmatrix}
		O&Q \\
		R&O
		\end{pmatrix}}\bigg],
	\end{align*}
where the last equality follows from Lemma \ref{p4200}. So, by applying \eqref{qq2} we get
\begin{align*}
&\omega_{\mathbb{A}}\left[\begin{pmatrix}
P&Q \\
R&S
\end{pmatrix}\right]\\
&\leq \frac{1}{2}\bigg[\omega_A(P)+\omega_A(S)+\sqrt{(\omega_A(P)-\omega_A(S))^2+(\omega_A(Q+R)+\omega_A(Q-R))^2}\bigg].
\end{align*}
This proves the desired inequality. Moreover, it can be observed that
\begin{align*}
& \max\{\omega_A(P), \omega_A(S)\}+\frac{\omega_A(Q+R)+\omega_A(Q-R)}{2}\\
 &=\frac{\omega_A(P)+\omega_A(S)+|\omega_A(P)-\omega_A(S)|}{2} + \frac{\omega_A(Q+R)+ \omega_A(Q-R)}{2}\\
&\geq \frac{1}{2}\bigg(\omega_A(P)+\omega_A(S)+\sqrt{(\omega_A(P)-\omega_A(S))^2+(\omega_A(Q+R)+\omega_A(Q-R))^2}\bigg).
\end{align*}
Hence, the proof is complete.
\end{proof}

As a consequence of Theorem \ref{them100} we state the following corollary.
\begin{corollary}\label{cor100}
Let  $P, Q, R, S\in \mathcal{B}_{A}(\mathcal{H})$. Then,
\begin{align*}
r_A(PQ+RS)
&\leq \frac{1}{2}\big[\omega_A(QP)+\omega_A(SR)\big]\\
&+\frac{1}{2}\sqrt{\big[\omega_A(QP)-\omega_A(SR)\big]^2+\big[\omega_A(QR+SP)+\omega_A(QR-SP)\big]^2}.
\end{align*}
\end{corollary}
In order to prove Corollary \ref{cor100}, we need the following lemma.
\begin{lemma}(\cite{feki01})\label{kkkk2020}
If $T\in \mathcal{B}_{A^{1/2}}(\mathcal{H})$, then
\begin{equation*}
r_A(T)\leq \omega_A(T).
\end{equation*}
\end{lemma}
Now we are in a position to establish Corollary \ref{cor100}.
\begin{proof}[Proof of Corollary \ref{cor100}]
It can be observed that
\begin{align*}
    r_A(PQ+RS)
    &=r_{\mathbb{A}}\left[\begin{pmatrix}
PQ+RS&O \\
O&O
\end{pmatrix}\right]\\
    &=r_{\mathbb{A}}\left[\begin{pmatrix}
P&R \\
O&O
\end{pmatrix}\begin{pmatrix}
Q&O \\
S&O
\end{pmatrix}\right]\\
&=r_{\mathbb{A}}\left[\begin{pmatrix}
Q&O \\
S&O
\end{pmatrix}\begin{pmatrix}
P&R \\
O&O
\end{pmatrix}\right]\quad (\text{by }\,\eqref{commut})\\
&=r_{\mathbb{A}}\left[\begin{pmatrix}
QP&QR \\
SP&SR
\end{pmatrix}\right]\\
&\leq \omega_{\mathbb{A}}\left[\begin{pmatrix}
QP&QR \\
SP&SR
\end{pmatrix}\right]\quad(\text{by Lemma } \ref{kkkk2020}).
\end{align*}
So, by applying  Theorem \ref{them100} we reach the required result.
\end{proof}
\begin{remark}
Recently the first author proved in \cite{feki02} that for $P, Q, R, S\in \mathcal{B}_{A}(\mathcal{H})$ it holds
\begin{align}\label{kk2020}
r_A(PQ+RS)
&\leq \frac{1}{2}\big[\omega_A(QP)+\omega_A(SR)\big]\nonumber\\
&+\frac{1}{2}\sqrt{\big[\omega_A(QP)-\omega_A(SR)\big]^2+4\left\Vert QR\right\Vert_A \left\Vert SP\right\Vert_A }.
\end{align}
If $QR=SP$, then it can be seen that the inequality in Corollary \ref{cor100} is sharper than \eqref{kk2020}.
\end{remark}

\begin{remark}
Notice that by letting $Q=S=I$ in Corollary \ref{cor100} we get
\begin{align*}
r_A(P+R)
&\leq \frac{1}{2}\big[\omega_A(P)+\omega_A(R)\big]\\
&+\frac{1}{2}\sqrt{(\omega_A(P)-\omega_A(R))^2+(\omega_A(P+R)+\omega_A(P-R))^2}.
\end{align*}
\end{remark}

To establish further upper bounds for the $\mathbb{A}$-numerical radius of the operator matrix $\begin{pmatrix}
P&Q \\
R&S
\end{pmatrix}$, we need the following lemmas.
\begin{lemma}\label{lmm05}(\cite{feki02})
Let $\mathbb{T}=\begin{pmatrix}
P&Q \\
R&S
\end{pmatrix}$ be such that $P, Q, R, S\in \mathcal{B}_{A^{1/2}}(\mathcal{H})$. Then, $\mathbb{T}\in \mathcal{B}_{\mathbb{A}^{1/2}}(\mathcal{H}\oplus \mathcal{H})$ and
$$\|\mathbb{T}\|_\mathbb{A}\leq \left\|\begin{pmatrix}
\|P\|_A & \|Q\|_A \\
\|R\|_A & \|S\|_A
\end{pmatrix}\right\|.$$
\end{lemma}
\begin{lemma}\label{a5so}
Let $T,S\in \mathcal{B}_{A}(\mathcal{H})$. Then,
$$\|T^{\sharp_A}S\|_A=\|S^{\sharp_A}T\|_A.$$
\end{lemma}
\begin{proof}
By using the fact that $P_{\overline{\mathcal{R}(A)}}A=AP_{\overline{\mathcal{R}(A)}}=A$ together with \eqref{newsemi}, we see that
\begin{align*}
\|T^{\sharp_A}S\|_A
&=\|S^{\sharp_A}P_{\overline{\mathcal{R}(A)}}TP_{\overline{\mathcal{R}(A)}}\|_A\\
&=\sup\left\{|\langle AP_{\overline{\mathcal{R}(A)}}x, (S^{\sharp_A}P_{\overline{\mathcal{R}(A)}}T)^{\sharp_A}y\rangle|\,;\;x,y\in \mathcal{H},\,\|x\|_{A}=\|y\|_{A}= 1\right\}\\
&=\sup\left\{|\langle S^{\sharp_A}P_{\overline{\mathcal{R}(A)}}Tx, y\rangle_A|\,;\;x,y\in \mathcal{H},\,\|x\|_{A}=\|y\|_{A}= 1\right\}\\
&=\sup\left\{|\langle AP_{\overline{\mathcal{R}(A)}}Tx, Sy\rangle|\,;\;x,y\in \mathcal{H},\,\|x\|_{A}=\|y\|_{A}= 1\right\}\\
&=\sup\left\{|\langle S^{\sharp_A}Tx, y\rangle_A|\,;\;x,y\in \mathcal{H},\,\|x\|_{A}=\|y\|_{A}= 1\right\}\\
&=\|S^{\sharp_A}T\|_A.
\end{align*}
This proves the desired result.
\end{proof}

Now, we are in a position to state the following theorem.
\begin{theorem}\label{theorem:upper 3}
Let $P,Q,R,S\in \mathcal{B}_{A}(\mathcal{H})$. Then,
\[\omega_\mathbb{A}\left[\begin{pmatrix}
P&Q \\
R&S
\end{pmatrix}\right]\leq \min\{\mu,\nu \},\]
where
\begin{align*}
\mu
 & =\frac{\sqrt{2}}{2}\sqrt{\|P\|_A^2+\|Q\|_A^2+\sqrt{(\|P\|_A^2-\|Q\|_A^2)^2 +4\|P^{\sharp_A}Q\|_A^2} } \\
 &\quad\quad\quad+\frac{\sqrt{2}}{2}\sqrt{\|R\|_A^2+\|S\|_A^2+\sqrt{(\|R\|_A^2-\|S\|_A^2)^2 +4\|S^{\sharp_A}R\|_A^2} },
\end{align*}
and
\begin{align*}
\nu
& =\frac{\sqrt{2}}{2}\sqrt{\|P\|_A^2+\|R\|_A^2+\sqrt{(\|P\|_A^2-\|R\|_A^2)^2+4\|PR^{\sharp_A}\|_A^2 }} \\
 &\quad\quad\quad+\frac{\sqrt{2}}{2}\sqrt{\|Q\|_A^2+\|S\|_A^2+\sqrt{(\|Q\|_A^2-\|S\|_A^2)^2 +4\|SQ^{\sharp_A}\|_A^2} }.
\end{align*}
\end{theorem}
\begin{proof}
We first prove that
\begin{equation}\label{ffirst}
\omega_\mathbb{A}\left[\begin{pmatrix}
P&Q \\
O&O
\end{pmatrix}\right]\leq \frac{\sqrt{2}}{2}\sqrt{\|P\|_A^2+\|Q\|_A^2+\sqrt{(\|P\|_A^2-\|Q\|_A^2)^2+4\|P^{\sharp_A}Q\|_A^2 }}\,.
\end{equation}
By using \eqref{refine1} together with \eqref{diez} we see that
\begin{align*}
 \omega_\mathbb{A}\left[\begin{pmatrix}
P&Q\\
    O& O
    \end{pmatrix}\right]
    &\leq \left\|\begin{pmatrix}
P&Q\\
    O & O
    \end{pmatrix}\right\|_\mathbb{A}\\
    &=\left\|\begin{pmatrix}
P&Q\\
    O& O
    \end{pmatrix}^{\sharp_\mathbb{A}}
    \begin{pmatrix}
P&Q\\
    O& O
    \end{pmatrix}\right\|_\mathbb{A}^{\frac{1}{2}}\\
    &=\left\|\begin{pmatrix}
    P^{\sharp_A}& O\\
    Q^{\sharp_A}& O
    \end{pmatrix}\begin{pmatrix}
    P&Q\\
    O& O
    \end{pmatrix}\right\|_\mathbb{A}^{\frac{1}{2}}\\
    &=\left\|\begin{pmatrix}
    P^{\sharp_A}P & P^{\sharp_A}Q\\
    Q^{\sharp_A}P& Q^{\sharp_A}Q
    \end{pmatrix}\right\|_\mathbb{A}^{\frac{1}{2}}.
\end{align*}
So, by using Lemma \ref{lmm05} together with Lemma \ref{a5so} we see that
\begin{align*}
 \omega_\mathbb{A}^2\left[\begin{pmatrix}
P&Q\\
    O& O
    \end{pmatrix}\right]
&\leq \left\|\begin{pmatrix}
    \|P^{\sharp_A}P\|_A & \|P^{\sharp_A}Q\|_A\\
    \|Q^{\sharp_A}P\|_A& \|Q^{\sharp_A}Q\|_A
    \end{pmatrix}\right\|\\
 &=\left\|\begin{pmatrix}
    \|P\|_A^2& \|P^{\sharp_A}Q\|_A\\
    \|P^{\sharp_A}Q\|_A& \|Q\|_A^2
    \end{pmatrix}\right\|\\
  &=r\left[\begin{pmatrix}
    \|P\|_A^2& \|P^{\sharp_A}Q\|_A\\
    \|P^{\sharp_A}Q\|_A& \|Q\|_A^2
    \end{pmatrix}\right]\\
  &= \frac{1}{2}\left(\|P\|_A^2+\|Q\|_A^2+\sqrt{(\|P\|_A^2-\|Q\|_A^2)^2+4\|P^{\sharp_A}Q\|_A^2}\right).
\end{align*}
This proves \eqref{ffirst}. Let $\mathbb{U}=\begin{pmatrix}
O&I \\
I&O
\end{pmatrix}.$ In view of Lemma \ref{lemma1} (iii) we have $\mathbb{U}\in \mathcal{B}_{\mathbb{A}}(\mathcal{H}\oplus \mathcal{H})$ and $\mathbb{U}^{\sharp_{\mathbb{A}}}=\begin{pmatrix}
O&P_{\overline{\mathcal{R}(A)}}\\
P_{\overline{\mathcal{R}(A)}}&O
\end{pmatrix}.$ Further, it can be seen that $\mathbb{U}$ is $\mathbb{A}$-unitary operator. So, by using Lemma \ref{weak} together with \eqref{ffirst} we get
\begin{align*}
\omega_{\mathbb{A}}\left[\begin{pmatrix}
P&Q \\
R&S
\end{pmatrix}\right]
& \leq\omega_{\mathbb{A}}\left[\begin{pmatrix}
P&Q\\
O &O
\end{pmatrix}\right]+\omega_{\mathbb{A}}\left[\begin{pmatrix}
O &O\\
R&S
\end{pmatrix}\right] \\
 &=\omega_{\mathbb{A}}\left[\begin{pmatrix}
P&Q\\
O &O
\end{pmatrix}\right]+\omega_{\mathbb{A}}\left[\mathbb{U}^{\sharp_{\mathbb{A}}}\begin{pmatrix}
O &O\\
R&S
\end{pmatrix}\mathbb{U}\right] \\
 &=\omega_{\mathbb{A}}\left[\begin{pmatrix}
P&Q\\
O &O
\end{pmatrix}\right]+\omega_{\mathbb{A}}\left[\begin{pmatrix}
P_{\overline{\mathcal{R}(A)}} &O\\
O &P_{\overline{\mathcal{R}(A)}}
\end{pmatrix}
\begin{pmatrix}
S &R\\
O &O
\end{pmatrix}\right]\\
 &=\omega_{\mathbb{A}}\left[\begin{pmatrix}
P&Q\\
O &O
\end{pmatrix}\right]+\omega_{\mathbb{A}}\left[
\begin{pmatrix}
S &R\\
O &O
\end{pmatrix}\right]\\
 & \leq \frac{\sqrt{2}}{2}\sqrt{\|P\|_A^2+\|Q\|_A^2+\sqrt{(\|P\|_A^2-\|Q\|_A^2)^2 +4\|P^{\sharp_A}Q\|_A^2} } \\
 &+\frac{\sqrt{2}}{2}\sqrt{\|R\|_A^2+\|S\|_A^2+\sqrt{(\|R\|_A^2-\|S\|_A^2)^2 +4\|S^{\sharp_A}R\|_A^2} }.
\end{align*}
Now, since $\omega_\mathbb{A}\left[\begin{pmatrix}
P&Q \\
R&S
\end{pmatrix}\right]=\omega_\mathbb{A}\left[\begin{pmatrix}
P^{\sharp_A} & R^{\sharp_A} \\
Q^{\sharp_A} & S^{\sharp_A}
\end{pmatrix}\right]$, then by using similar arguments as above we obtain
\begin{align*}
&\omega_\mathbb{A}\left[\begin{pmatrix}
P&Q \\
R&S
    \end{pmatrix}\right]\\
& \leq\frac{\sqrt{2}}{2}\sqrt{\|P^{\sharp_A}\|_A^2+\|R^{\sharp_A}\|_A^2+\sqrt{(\|P^{\sharp_A}\|_A^2-\|R^{\sharp_A}\|_A^2)^2+4\|(P^{\sharp_A})^{\sharp_A}R^{\sharp_A}\|_A^2 }} \\
 &+\frac{\sqrt{2}}{2}\sqrt{\|Q^{\sharp_A}\|_A^2+\|S^{\sharp_A}\|_A^2+\sqrt{(\|Q^{\sharp_A}\|_A^2-\|S^{\sharp_A}\|_A^2)^2 +4\|(S^{\sharp_A})^{\sharp_A}Q^{\sharp_A}\|_A^2} }\\
 & =\frac{\sqrt{2}}{2}\sqrt{\|P\|_A^2+\|R\|_A^2+\sqrt{(\|P\|_A^2-\|R\|_A^2)^2+4\|PR^{\sharp_A}\|_A^2 }} \\
 &+\frac{\sqrt{2}}{2}\sqrt{\|Q\|_A^2+\|S\|_A^2+\sqrt{(\|Q\|_A^2-\|S\|_A^2)^2 +4\|SQ^{\sharp_A}\|_A^2} }\,.
\end{align*}
Hence, the proof is complete.
\end{proof}

Our next result provides a lower bound for the $\mathbb{A}$-numerical radius of a $2 \times 2$ operator matrix such that its second row consists of zero operators. We mention here that our result improves a recent result proved by Rout et al. in \cite{rout} since $\mathcal{B}_{A}(\mathcal{H})\subseteq\mathcal{B}_{A^{1/2}}(\mathcal{H})$.
\begin{theorem}\label{Theorem 2.6}
	Let $P, Q\in\mathcal{B}_{A^{1/2}}(\mathcal{H})$. Then
\begin{equation}\label{sahoo1}
\omega_{\mathbb{A}}\left[\begin{pmatrix}
	P & Q\\
	O & O
	\end{pmatrix}\right]\geq \frac{1}{2}\max\{\alpha,\beta\},
\end{equation}
where $\alpha=\omega_A(P+ Q)+\omega_A(P- Q)$ and $\beta=\omega_A(P+ iQ)+\omega_A(P- iQ)$.
\end{theorem}
\begin{proof}
	Let $\mathbb{U}=\begin{pmatrix}
	O & I\\I & O
	\end{pmatrix}$. 	
	By using Lemma \ref{lemma1} (iii), we see that
	$$\mathbb{U}^{\sharp_\mathbb{A}}=\begin{pmatrix}
	P_{\overline{\mathcal{R}(A)}} &O\\
	O &P_{\overline{\mathcal{R}(A)}}
	\end{pmatrix}\mathbb{U}=\begin{pmatrix}
	O & P_{\overline{\mathcal{R}(A)}}\\
	P_{\overline{\mathcal{R}(A)}} & O
	\end{pmatrix}.$$
	So, by using the fact that $AP_{\overline{\mathcal{R}(A)}}=A$, we can verify that $\|\mathbb{U}x\|_{\mathbb{A}}=\|\mathbb{U}^{\sharp_{\mathbb{A}}}x\|_{\mathbb{A}}=\|x\|_{\mathbb{A}}$ for all $x\in \mathcal{H}\oplus \mathcal{H}$. Hence $\mathbb{U}$ is $\mathbb{A}$-unitary. So, using  Lemma  \ref{lem100} (i) we observes that
	\begin{align*}
	\max\{\omega_A(P+Q),\omega_A(P-Q)\}=&\omega_{\mathbb{A}}\left[\begin{pmatrix}
	P&Q\\
	Q &P
	\end{pmatrix}\right]\\
	& \leq\omega_{\mathbb{A}}\left[\begin{pmatrix}
	P&Q\\
	O &O
	\end{pmatrix}\right]+\omega_{\mathbb{A}}\left[\begin{pmatrix}
	O &O\\
	Q&P
	\end{pmatrix}\right] \\
	&=\omega_{\mathbb{A}}\left[\begin{pmatrix}
	P&Q\\
	O &O
	\end{pmatrix}\right]+\omega_{\mathbb{A}}\left[\mathbb{U}^{\sharp_{\mathbb{A}}}\begin{pmatrix}
	O &O\\
	Q&P
	\end{pmatrix}\mathbb{U}\right] \\
	&=\omega_{\mathbb{A}}\left[\begin{pmatrix}
	P&Q\\
	O &O
	\end{pmatrix}\right]+\omega_{\mathbb{A}}\left[\begin{pmatrix}
	P_{\overline{\mathcal{R}(A)}} &O\\
	O &P_{\overline{\mathcal{R}(A)}}
	\end{pmatrix}
	\begin{pmatrix}
	P &Q\\
	O &O
	\end{pmatrix}\right]\\
	&=\omega_{\mathbb{A}}\left[\begin{pmatrix}
	P&Q\\
	O &O
	\end{pmatrix}\right]+\omega_{\mathbb{A}}\left[
	\begin{pmatrix}
	P &Q\\
	O &O
	\end{pmatrix}\right]\\
	=&2\omega_{\mathbb{A}}\left[
	\begin{pmatrix}
	P &Q\\
	O &O
	\end{pmatrix}\right].
	\end{align*}

On the other hand, by considering the $\mathbb{A}$-unitary operator $\mathbb{V}=\begin{pmatrix}
	I & O\\
O & -I
	\end{pmatrix}$ and proceeding as above we see that
	\begin{align*}
	\max\{\omega_A(P+iQ),\omega_A(P-iQ)\}=&\omega_{\mathbb{A}}\left[\begin{pmatrix}
	P&-Q\\
	Q &P
	\end{pmatrix}\right]\quad(\text{by Lemma } \ref{lem100}\;(ii))\\
	&\leq 2\omega_{\mathbb{A}}\left[
	\begin{pmatrix}
	P &Q\\
	O &O
	\end{pmatrix}\right].
	\end{align*}
	Hence we get our desired result.
\end{proof}

Now, in order to prove a lower bound for $\omega_\mathbb{A}\left[\begin{pmatrix}
	P & Q\\
	O & O
	\end{pmatrix}\right]$, we need the following lemma.
\begin{lemma}(\cite{feki004})\label{n1}
Let $T, S\in\mathcal{B}_{A}(\mathcal{H})$. Then,
\begin{align}\label{SK1}
\omega_A(TS\pm ST^{\sharp_A}) \leq2\|T\|_A\,\omega_A(S).
\end{align}
\end{lemma}

In the next result, Lemma \ref{lem100} enables us to present another application of the inequality \eqref{SK1}.
\begin{theorem}\label{app1}
	Let $P, Q\in\mathcal{B}_A(\mathcal{H})$. Then
\begin{equation}\label{sahoo3}
\omega_\mathbb{A}\left[\begin{pmatrix}
	P & Q\\
	O & O
	\end{pmatrix}\right]\geq \frac{1}{2}\max\left\{\omega_A(P+ iQ),\omega_A(P- iQ)\right\}.
\end{equation}
\end{theorem}
\begin{proof}
Let $\mathbb{T}=\begin{pmatrix}
	P & Q\\
	O & O
	\end{pmatrix}$ and $\mathbb{U}=\begin{pmatrix}
	O & -I\\I & O
	\end{pmatrix}$. It is not difficult to verify that $\mathbb{U}$ is $\mathbb{A}$-unitary. Moreover, since $P_{\overline{\mathcal{R}(A)}}X^{\sharp_A}=X^{\sharp_A}P_{\overline{\mathcal{R}(A)}}=X^{\sharp_A}$ for all $X\in\mathcal{B}_A(\mathcal{H})$ (see \cite{faiot}), then it can be verified that
$$\mathbb{U}^{\sharp_{\mathbb{A}}}\mathbb{T}^{\sharp_{\mathbb{A}}}+\mathbb{T}^{\sharp_{\mathbb{A}}}(\mathbb{U}^{\sharp_{\mathbb{A}}})^{\sharp_{\mathbb{A}}}=\begin{pmatrix}
	Q^{\sharp_A} & P^{\sharp_A}\\
	-P^{\sharp_A} & Q^{\sharp_A}
	\end{pmatrix} =\begin{pmatrix}
	Q & -P\\
	P & Q
	\end{pmatrix}^{\sharp_{\mathbb{A}}}.$$
By using Lemma \ref{n1} we see that
$$\omega_{\mathbb{A}}(\mathbb{U}^{\sharp_{\mathbb{A}}}\mathbb{T}^{\sharp_{\mathbb{A}}}+\mathbb{T}^{\sharp_{\mathbb{A}}}(\mathbb{U}^{\sharp_{\mathbb{A}}})^{\sharp_{\mathbb{A}}})
\leq 2\omega_{\mathbb{A}}(\mathbb{T}^{\sharp_{\mathbb{A}}}),$$
which, in turn, implies that
\begin{align*}
\omega_{\mathbb{A}}\left[\begin{pmatrix}
	P & Q\\
	O & O
	\end{pmatrix} \right]
&\geq \frac{1}{2}\omega_{\mathbb{A}}\left[\begin{pmatrix}
	Q & -P\\
	P & Q
	\end{pmatrix}^{\sharp_{\mathbb{A}}} \right] \\
&=\frac{1}{2}\omega_{\mathbb{A}}\left[\begin{pmatrix}
	Q & -P\\
	P & Q
	\end{pmatrix}\right] \\
 &=\frac{1}{2}\max\left\{\omega_A(P+ iQ),\omega_A(P- iQ)\right\},
\end{align*}
where the last equality follows by applying Lemma \ref{lem100}(ii).
\end{proof}

As an application of the above theorem, we can derive the following $A$-numerical radius inequality.
\begin{theorem}\label{f12}
Let $T\in \mathcal{B}_A(\mathcal{H})$. Then
$$\omega_A(T)\leq 2\min\left\{\omega_\mathbb{A}\left[\begin{pmatrix}
	\Re_A(T) & O\\
	\Im_A(T) & O
	\end{pmatrix}\right], \omega_\mathbb{A}\left[\begin{pmatrix}
	O & -i\Im_A(T)\\
	\Re_A(T) & O
	\end{pmatrix}\right] \right\}. $$
\end{theorem}
To prove Theorem \ref{f12}, we need the following lemma.
\begin{lemma}\label{lemf}
Let $T,S\in \mathcal{B}_A(\mathcal{H})$. Then
\begin{align}
 \omega_A(T\pm iS)\leq 2\,\omega_{\mathbb{A}}\left[\begin{pmatrix}
O&T\\
iS &O
\end{pmatrix}\right].
\end{align}
\end{lemma}
\begin{proof}
Let $\mathbb{X}=\begin{pmatrix}
I &I\\
O &O
\end{pmatrix}$ and $\mathbb{Y}=\begin{pmatrix}
O &T\\
S &O
\end{pmatrix}$. It can be observed that
$$\mathbb{X}\mathbb{Y}\mathbb{X}^{\sharp_\mathbb{A}}=\begin{pmatrix}
TP_{\overline{\mathcal{R}(A)}}+ SP_{\overline{\mathcal{R}(A)}} &O\\
O &O
\end{pmatrix},$$
and
\begin{align*}
\|\mathbb{X}\|_{\mathbb{A}}^2=\|\mathbb{X}\mathbb{X}^{\sharp_\mathbb{A}}\|_{\mathbb{A}}
&=\left\|\begin{pmatrix}
2P_{\overline{\mathcal{R}(A)}}&O\\
O &O
\end{pmatrix}\right\|_{\mathbb{A}}\\
 &=2\|P_{\overline{\mathcal{R}(A)}}\|_A=2.
\end{align*}
So, by using the fact that $AP_{\overline{\mathcal{R}(A)}}=A$ we see that
\begin{align*}\label{inff}
\omega_A(T+S)
&=\omega_A(TP_{\overline{\mathcal{R}(A)}}+ SP_{\overline{\mathcal{R}(A)}})\\
&=\omega_\mathbb{A}\left[\begin{pmatrix}
TP_{\overline{\mathcal{R}(A)}}+ SP_{\overline{\mathcal{R}(A)}} &O\\
O &O
\end{pmatrix}\right]\\
& = \omega_\mathbb{A}(\mathbb{X}\mathbb{Y}\mathbb{X}^{\sharp_\mathbb{A}})\quad(\text{by Lemma } \ref{lemma1})\nonumber\\
 &\leq \|\mathbb{X}\|_{\mathbb{A}}^2\omega_\mathbb{A}(\mathbb{Y})\quad(\text{by \cite[Lemma 4.4.]{zamani1}})\nonumber\\
  &= 2\omega_\mathbb{A}(\mathbb{Y}).
\end{align*}
This gives
\begin{equation}\label{c2}
\omega_A(T+S)\leq2\omega_\mathbb{A}\left[\begin{pmatrix}
O &T\\
S &O
\end{pmatrix}\right].
\end{equation}
By replacing $S$ by $-S$ in \eqref{c2} and then using the fact that $\omega_\mathbb{A}\left[\begin{pmatrix}
O &T\\
S &O
\end{pmatrix}\right]=\omega_\mathbb{A}\left[\begin{pmatrix}
O &T\\
-S &O
\end{pmatrix}\right]$ we get
\begin{equation}\label{5arita}
 \omega_A(T\pm S)\leq 2\,\omega_{\mathbb{A}}\left[\begin{pmatrix}
O&T\\
S &O
\end{pmatrix}\right].
\end{equation}
Finally, by replacing $S$ by $iS$ in \eqref{5arita}, we reach the desired results.
\end{proof}
Now we are ready to prove Theorem \ref{f12}.

\begin{proof}[Proof of Theorem \ref{f12}]
Clearly $T$ can written as $T=\Re_A(T)+i\Im_A(T)$ where
$$\Re_A(T):=\frac{T+T^{\sharp_A}}{2}\;\;\text{ and }\;\;\Im_A(T):=\frac{T-T^{\sharp_A}}{2i}.$$
So, $T^{\sharp_A}=[\Re_A(T)]^{\sharp_A}-i[\Im_A(T)]^{\sharp_A}$. Moreover, a short calculation reveals that $(T^{\sharp_A})^{\sharp_A}=[\Re_A(T)]^{\sharp_A}+i[\Im_A(T)]^{\sharp_A}$. So, by applying Theorem \ref{app1} we see that
\begin{align*}
&\omega_\mathbb{A}\left[\begin{pmatrix}
	\Re_A(T) & O\\
	\Im_A(T) & O
	\end{pmatrix}\right]\\
&=\omega_\mathbb{A}\left[\begin{pmatrix}
	[\Re_A(T)]^{\sharp_A} & [\Im_A(T)]^{\sharp_A}\\
	O & O
	\end{pmatrix}\right]\\
&\geq \frac{1}{2}\max\left\{\omega_A\Big([\Re_A(T)]^{\sharp_A}+ i[\Im_A(T)]^{\sharp_A}\Big),\omega_A\Big([\Re_A(T)]^{\sharp_A}- i[\Im_A(T)]^{\sharp_A}\Big)\right\}\\
&=\frac{1}{2}\max\left\{\omega_A(T^{\sharp_A}),\omega_A\left((T^{\sharp_A})^{\sharp_A}\right)\right\}\\
&=\frac{1}{2}\omega_A(T).
\end{align*}
On the other hand, by applying Lemma \ref{lemf} we get
\begin{align*}
&\omega_\mathbb{A}\left[\begin{pmatrix}
	O & -i\Im_A(T)\\
	\Re_A(T) & O
	\end{pmatrix}\right]\\
&=\omega_\mathbb{A}\left[\begin{pmatrix}
	O & [\Re_A(T)]^{\sharp_A}\\
	i[\Im_A(T)]^{\sharp_A} & O
	\end{pmatrix}\right]\\
&\geq \frac{1}{2}\max\left\{\omega_A\Big([\Re_A(T)]^{\sharp_A}+ i[\Im_A(T)]^{\sharp_A}\Big),\omega_A\Big([\Re_A(T)]^{\sharp_A}- i[\Im_A(T)]^{\sharp_A}\Big)\right\}\\
&=\frac{1}{2}\omega_A(T).
\end{align*}
\end{proof}

Based on Lemma \ref{lem100}, the forthcoming theorem which provides a lower bound for $\mathbb{A}$-numerical radius of a $2 \times 2$ operator matrix is also an application of the inequality \eqref{SK1}.
\begin{theorem}
	Let $P, Q,R,S\in\mathcal{B}_A(\mathcal{H})$. Then
\begin{equation}\label{sahoo2}
\omega_\mathbb{A}\left[\begin{pmatrix}
	P & Q\\
	R & S
	\end{pmatrix}\right]\geq \frac{1}{2}\max\{\mu,\nu\},
\end{equation}
where $$\mu=\max\{\omega_A(Q+R+S+P) ,\omega_A(Q+R-S-P) \},$$ and $$\nu=\max\big\{\omega_A\big(Q-R+i(S+P)\big) ,\omega_A\big(Q-R-i(S+P)\big)\big\}.$$
\end{theorem}
\begin{proof}
	Let $\mathbb{U}=\begin{pmatrix}
	O & I\\I & O
	\end{pmatrix}$. 	
	By using Lemma \ref{lemma1} (iii), we see that
	$$\mathbb{U}^{\sharp_\mathbb{A}}=\begin{pmatrix}
	P_{\overline{\mathcal{R}(A)}} &O\\
	O &P_{\overline{\mathcal{R}(A)}}
	\end{pmatrix}\mathbb{U}=\begin{pmatrix}
	O &P_{\overline{\mathcal{R}(A)}}\\
	P_{\overline{\mathcal{R}(A)}} &O
	\end{pmatrix}.$$
	So, by using the fact that $AP_{\overline{\mathcal{R}(A)}}=A$, we can verify that $\|\mathbb{U}x\|_{\mathbb{A}}=\|\mathbb{U}^{\sharp_{\mathbb{A}}}x\|_{\mathbb{A}}=\|x\|_{\mathbb{A}}$ for all $x\in \mathcal{H}\oplus \mathcal{H}$. Hence $\mathbb{U}$ is $\mathbb{A}$-unitary. Moreover, clearly $(\mathbb{U}^{\sharp_{\mathbb{A}}})^{\sharp_{\mathbb{A}}}=\mathbb{U}^{\sharp_{\mathbb{A}}}$. Now, let $\mathbb{T}=\begin{pmatrix}
	P & Q\\
	R & S
	\end{pmatrix}$. Since, $P_{\overline{\mathcal{R}(A)}}X^{\sharp_A}=X^{\sharp_A}P_{\overline{\mathcal{R}(A)}}=X^{\sharp_A}$ for all $X\in\mathcal{B}_A(\mathcal{H})$ (see \cite{faiot}), then a short calculation shows that
$$\mathbb{U}^{\sharp_{\mathbb{A}}}\mathbb{T}^{\sharp_{\mathbb{A}}}+\mathbb{T}^{\sharp_{\mathbb{A}}}(\mathbb{U}^{\sharp_{\mathbb{A}}})^{\sharp_{\mathbb{A}}}=\begin{pmatrix}
	Q^{\sharp_A}+R^{\sharp_A} & S^{\sharp_A}+P^{\sharp_A}\\
	S^{\sharp_A}+P^{\sharp_A} & Q^{\sharp_A}+R^{\sharp_A}
	\end{pmatrix} =\begin{pmatrix}
	Q+R & S+P\\
	S+P & Q+R
	\end{pmatrix}^{\sharp_{\mathbb{A}}}.$$
By using Lemma \ref{n1}, we see that
$$\omega_{\mathbb{A}}(\mathbb{U}^{\sharp_{\mathbb{A}}}\mathbb{T}^{\sharp_{\mathbb{A}}}+\mathbb{T}^{\sharp_{\mathbb{A}}}(\mathbb{U}^{\sharp_{\mathbb{A}}})^{\sharp_{\mathbb{A}}})
\leq 2\omega_{\mathbb{A}}(\mathbb{T}^{\sharp_{\mathbb{A}}}),$$
which, in turn, implies that
\begin{align*}
\omega_{\mathbb{A}}\left[\begin{pmatrix}
	P & Q\\
	R & S
	\end{pmatrix} \right]
&\geq \frac{1}{2}\omega_{\mathbb{A}}\left[\begin{pmatrix}
	Q+R & S+P\\
	S+P & Q+R
	\end{pmatrix} \right] \\
 &=\frac{1}{2}\max\{\omega_A(Q+R+S+P) ,\omega_A(Q+R-S-P) \}:=\frac{1}{2}\mu,
\end{align*}
where the last equality follows by applying Lemma \ref{lem100} (i). On the other hand, by choosing $\mathbb{U}=\begin{pmatrix}
	O & I\\-I & O
	\end{pmatrix}$ and proceeding as above we obtain
\begin{align*}
\omega_{\mathbb{A}}\left[\begin{pmatrix}
	P & Q\\
	R & S
	\end{pmatrix} \right]
&\geq \frac{1}{2}\omega_{\mathbb{A}}\left[\begin{pmatrix}
	Q-R & -(S+P)\\
	S+P & Q-R
	\end{pmatrix} \right] \\
 &=\frac{1}{2}\max\big\{\omega_A\big(Q-R+i(S+P)\big) ,\omega_A\big(Q-R-i(S+P)\big)\big\}:=\frac{1}{2}\nu,
\end{align*}
where the last equality follows by applying Lemma \ref{lem100} (ii). This completes the proof of the theorem.
\end{proof}

\end{document}